\documentclass[12 pt,english,oneside]{amsart}
\usepackage[top=1.26in, bottom=1.26in, left=1.33in, right=1.33in]{geometry}
\usepackage[latin1]{inputenc}
\usepackage[T1]{fontenc}
\usepackage[normalem]{ulem}
\usepackage{verbatim} 
\usepackage{graphicx}
\usepackage{xypic} 
\usepackage{amsfonts}
\usepackage{amsmath}
\usepackage{amssymb}
\usepackage{graphicx,epsfig}
\usepackage{amsthm}
\usepackage{amscd}
\usepackage{mathrsfs}
\usepackage{dsfont}
\usepackage{fancyhdr}
\usepackage{enumerate}

\newcommand{\g}{\mathfrak{g}}
\newcommand{\uts}{\underline{\ts}}

\newcommand{\ra}{\rightarrow}
\newcommand{\ve}{\varepsilon}
\newcommand{\vp}{\varphi}
\newcommand{\M}{\mathbf{M}}

\newcommand{\ts}{\otimes}
\newcommand{\s}{\sigma}

\newcommand{\ff}{\mathcal{F}}

\newcommand{\mc}{\mathbb{K}}

\newcommand{\wt}{\widetilde}

\newcommand{\id}{\operatorname*{id}}
\newcommand{\ad}{\operatorname*{ad}}
\newcommand{\hht}{\operatorname*{ht}}
\newcommand{\corad}{\operatorname*{corad}}
\newcommand{\Hom}{\operatorname*{Hom}}
\newcommand{\End}{\operatorname*{End}}

\newtheorem{definition}{Definition}
\newtheorem{proposition}{Proposition}
\newtheorem{theorem}{Theorem}
\newtheorem{lemma}{Lemma}
\newtheorem{corollary}{Corollary}
\newtheorem{remark}{Remark}

\newtheorem{problem}{Problem}

\title{Multi-brace cotensor Hopf algebras and quantum groups}
\author{Xin Fang and Marc Rosso}
\address{Universit\'e Paris Diderot - Paris VII, UFR de Math\'ematiques, Case 7012, B\^atiment Chevaleret, 75205 Paris Cedex 13, FRANCE.}
\email{fang@math.jussieu.fr, rosso@math.jussieu.fr}

\begin{document}

\maketitle

\begin{abstract}
We construct multi-brace cotensor Hopf algebras with bosonizations of quantum multi-brace algebras as examples. Quantum quasi-symmetric algebras are then obtained by taking particular initial data; this allows us to realize the whole quantum group associated to a symmetrizable Kac-Moody Lie algebra as a quantum quasi-symmetric algebra and all highest weight irreducible representations can be constructed using this machinery. It also provides a systematic way to construct simple modules over the quantum double of a quantum group.
\end{abstract}

\section{Introduction}
The work in this paper arose from two different motivations.

\subsection{First motivation: classifying algebra structures}
The tensor space $T(V)$ associated with a vector space $V$ is a connected coalgebra with the deconcatenation coproduct. All admissible associative products compatible with this coproduct are classified by Loday and Ronco in \cite{LR06} by investigating associativity restrictions on linear maps $\M_{pq}:V^{\ts p}\ts V^{\ts q}\ra V$. This yields a connected bialgebra $T(V)$.
\par
In the eighties of the last century, quantized enveloping algebras (quantum groups) are constructed in the work of Drinfel'd and Jimbo with the aim of finding explicit solutions of Yang-Baxter equations: such a solution gives rise to a vector space $V$ with a braiding $\s\in GL(V\ts V)$, which is called a braided vector space.
\par
Given a braided vector space $(V,\s)$ and the coalgebra $T(V)$ endowed with the deconcatenation coproduct, Jian and the second named author \cite{JR12} classified all associative products on $T(V)$ to yield a braided bialgebra: besides a braided version of associativity restrictions on $\M_{pq}:V^{\ts p}\ts V^{\ts q}\ra V$, these maps must be compatible with the braiding $\s$ (Definition 4.9, \textit{loc.cit}).
\par
The coalgebra $T(V)$ appearing in the two constructions above is a particular case of the cotensor coalgebra $T_H^c(M)$ of an $H$-Hopf bimodule $M$ on a Hopf algebra $H$ constructed by Nichols \cite{Nic78}. To recover the framework of Loday and Ronco, it suffices to take $H=\mc$ be the base field with the trivial Hopf algebra structure. In general, the set of right coinvariants $M^{coR}$ in $M$ admits a braiding arising from the $H$-Yetter-Drinfel'd module structure on it: this gives naturally a braided vector space. The following problem asks for a generalization of the results quoted above:

\begin{problem}\label{Pb1}
Classify all associative products on the coalgebra $T_H^c(M)$ yielding a Hopf algebra.
\end{problem}

\subsection{Second motivation: construction of quantum groups}
One of the central problems in the theory of quantum groups is to find different realizations of them. After the original definition using generators and relations due to Drinfel'd and Jimbo, Ringel \cite{Rin90} and Green \cite{Gre95} found the  negative (or positive) part of a (specialized) quantum group inside the Hall algebra of the category of quiver representations; this motivates the construction of Lusztig \cite{Lus91} using perverse sheaves on quiver varieties. The other construction of a half of the quantum group is given by the second named author using quantum shuffle algebras \cite{Ros98}.
\par
These works largely promote new developments in representation theory and in Hopf algebra theory: the former gives the existence of the canonical bases and the latter opens the gate to the classification of finite dimensional pointed Hopf algebras. A large quantity of works are dedicated to the study of these new features from different view points.
\par
It should be remarked that a half of the quantum coordinate algebra associated with a quantum group is realized recently by Gei\ss, Leclerc and Schr\"oer \cite{GLS11} as a quantum cluster algebra in a quantum torus.
\par
All these constructions above concern only half of a quantum group. It is therefore natural to ask for a construction of the whole quantum group parallel to the two constructions above.
\par
In a recent preprint \cite{Bri11}, Bridgeland gives a first successful attempt to construct the whole quantum group using Hall algebras: he used the category of $\mathbb{Z}/2$-graded complexes formed by quiver representations to imitate the positive and negative parts and restricted the morphisms between the odd-graded and even-graded parts to obtain the correct commutation relation between them. As this construction has its roots in the representation theory of quivers, only symmetrized Cartan matrices are dealt with.

\begin{problem}\label{Pb2}
Give a construction of the whole quantum group in spirit of (quantum) shuffle algebras.
\end{problem}

One of the main objectives of this paper is to give answers to these problems.

\subsection{Multi-brace cotensor Hopf algebras and quantum symmetric algebras}

According to the universal property of the cotensor coalgebra $T_H^c(M)$ due to Nichols \cite{Nic78}, a coalgebra map $T_H^C(M)\ts T_H^C(M)\ra T_H^C(M)$ is uniquely determined by its projections onto degree 0: $g:T_H^C(M)\ts T_H^C(M)\ra H$ and degree 1: $f:T_H^C(M)\ts T_H^C(M)\ra M$. We define the property (\textbf{MB}) for the pair $(f,g)$ and show that such pairs produce associative products on the cotensor coalgebra to yield a Hopf algebra structure. This solves Problem \ref{Pb1}.
\par
This construction generalizes quantum multi-brace algebras defined in \cite{JR12}: we introduce the notion of a graded pair characterized by the property that the inclusion $H\ra T_H^C(M)$ and the projection $T_H^C(M)\ra H$ fit into the framework of "Hopf algebra with a projection" due to Radford \cite{Rad}. If the pair $(f,g)$ is graded, the set of right coinvariants with respect to the right $H$-Hopf module structure on $T_H^C(M)$ admits a quantum multi-brace algebra structure. In this case, the multi-brace cotensor Hopf algebra is a bosonization of the corresponding quantum multi-brace algebra.
\par
Some particular cases of these algebras are of great interest.
\begin{enumerate}
\item If $(f,g)$ is graded and $f$ concentrates on bidegree $(0,1)$ and $(1,0)$ where it is given by the $H$-module structural maps on $M$, the product is the quantum symmetric algebra \cite{Ros98} (or equivalently, bosonization of a Nichols algebra \cite{AS02}).
\item If it is required moreover in (1) that $f$ is not zero on bidegree $(1,1)$, the result is a bosonization of the quantum quasi-shuffle algebra studied in \cite{JRZ11} and \cite{J12} as a quantization of the quasi-shuffle algebra \cite{NR79}; it is called quantum quasi-symmetric algebra. 
\end{enumerate}

Interesting examples of quantum quasi-shuffle algebras such as Rota-Baxter algebras and tridendriform algebras are constructed in \cite{J12}.

\subsection{Construction of quantum groups and their representations}

The quantum quasi-symmetric algebra enables us to construct the whole quantum group: we introduce some "dummy parameters" $\{\xi_i\}$ as the bracket of the generators in the positive and negative parts and then specialize them to the torus part. The choice of specialization seems arbitrary, but if it is demanded that the structure of Hopf algebra should not be destroyed, we have (up to scalar) a unique candidate: the right hand side in the commutation relation of quantum groups!
\par
The advantage of this construction is at least twofold: we could start with any symmetrizable generalized Cartan matrices, and all highest weight irreducible representations can be naturally constructed as the set of coinvariants by considering a Radford pair (see Theorem \ref{Thm:Rep}). Moreover, this machinery produces not only simple modules over quantum groups but also simple modules over the quantum double of a quantum group (the double of a double). This gives natural explanations of some results due to Baumann and Schmidt \cite{BS98} and Joseph \cite{Jos99}.

\subsection{Organization of this paper}
Hopf algebra structures on a given cotensor coalgebra are classified in Section \ref{Sec1} and examples related to Nichols algebras, quantum multi-brace algebras and quantum quasi-symmetric algebras are discussed in Section \ref{Sec2}. Quantum groups associated with a symmetrizable Kac-Moody Lie algebra are constructed as specializations of quantum quasi-symmetric algebras in Section \ref{Sec:main} and highest weight irreducible representations are obtained by considering a Radford pair in Section \ref{Sec4}, where modules over the quantum double of a quantum group are studied. Finally we prove the uniqueness of the specialization in Section \ref{Sec5}.

\subsection{Notations}
In this paper, we let $\mc$ denote a field of characteristic 0.

\section{Pairs of property (\textbf{MB})}\label{Sec1}

\subsection{Cotensor coalgebras}
The cotensor product of two bicomodules over a coalgebra $C$ is a dual construction of the tensor product over an algebra. 
\par
Let $C$ be a coalgebra and $M,N$ be two $C$-bicomodules with structural maps $\delta_L$ and $\delta_R$. The cotensor product of $M$ and $N$ is a $C$-bicomodule defined as follows: we consider two linear maps $\delta_R\ts \id_N$, $\id_M\ts \delta_L:\ M\ts N\ra M\ts C\ts N$; the cotensor product of $M$ and $N$, denoted by $M\Box_C N$, is the equalizer of $\delta_R\ts \id_N$ and $\id_M\ts \delta_L$.
\par
The cotensor coalgebra $T_C^c(M)$ is defined as a graded vector space with $T_C^c(M)_0=C$, $T_C^c(M)_1=M$ and $T_C^c(M)_k=\Box^k_CM$ for $k>1$. The coproduct is graded and is defined as follows: 
$$\Delta:\Box^n_CM\ra \sum_{i+j=n}\Box^i_CM\ts \Box^j_CM,$$ 
whose $(i,j)$-component is given by: for $i=0$ or $j=0$, it is the bicomodule structure map; for $i,j>0$, it is induced by the deconcatenation coproduct 
$$m_1\ts\cdots\ts m_n\mapsto  (m_1\ts\cdots\ts m_i)\ts (m_{i+1}\ts\cdots\ts m_n).$$
We let $\pi:T_C^c(M)\ra C$ and $p:T_C^c(M)\ra M$ denote the projection onto degree $0$ and $1$, respectively. Then the counit of $T_C^c(M)$ is $\ve_C\circ\pi$.
\par
The main construction of the multi-brace cotensor Hopf algebra comes from the following universal property of cotensor coalgebras:

\begin{proposition}[\cite{Nic78}]\label{Prop:1}
Let $C,D$ be two coalgebras, $g:D\ra C$ be a coalgebra morphism, $M$ be a $C$-bicomodule and $f:D\ra M$ be a $C$-bicomodule morphism such that $f(\corad(D))=0$. Then there exists a unique coalgebra map $F:D\ra T_C^c(M)$ such that the following two diagrams commute:
$$\xymatrix{
  D \ar[dr]_-{g} \ar[r]^-{F}
                & T_C^c(M) \ar[d]^-{\pi}  \\
                & C            },\ \ \ \ 
      \xymatrix{
  D \ar[dr]_-{f} \ar[r]^-{F}
                & T_C^c(M) \ar[d]^-{p}  \\
                & M            }.$$

\end{proposition}

It should be remarked that $D$ is a $C$-bicomodule with structural maps $(g\ts {\id})\Delta$ and $({\id}\ts g)\Delta$. Moreover, the coalgebra map $F$ has the following explicit form:
$$F=g+\sum_{n\geq 1}f^{\ts n}\Delta^{(n)},$$
where we use the notation that $\Delta^{(1)}={\id}$ and $\Delta^{(n)}=({\id}^{\ts n-2}\ts\Delta)\Delta^{(n-1)}$.

\begin{remark}
In fact, a direct verification shows that for any $n=1,\cdots, i+j$, $f^{\ts n}\Delta^{(n)}:M^{\Box i}\ts M^{\Box j}\ra M^{\ts n}$ has image in $M^{\Box n}$. 
\end{remark}

\subsection{Property (\textbf{MB})}
We will adopt the notation $\uts$ when the tensor product is between two tensor algebras.
\par
Let $H$ be a Hopf algebra and $M$ be an $H$-Hopf bimodule. We construct algebra structures and determine all these structures on $T_H^c(M)$ compatible with the deconcatenation coproduct. To be more precise, we take $D=T_H^c(M)\underline{\ts} T_H^c(M)$ and ask under which conditions on $f$ and $g$, the coalgebra map given in Proposition \ref{Prop:1}
$$F:T_H^c(M)\underline{\ts} T_H^c(M)\ra T_H^c(M)$$ 
is associative.

As a vector space, $$T_H^c(M)=H\oplus \bigoplus_{n>0}M^{\Box n},$$
then $T_H^c(M)\underline{\ts} T_H^c(M)$ decomposes into the following sum:
$$T_H^c(M)\underline{\ts} T_H^c(M)=(H\underline{\ts} H)\oplus\bigoplus_{n>0}\left(\left(H\underline{\ts} M^{\Box n}\right)\oplus\left(M^{\Box n}\underline{\ts} H\right)\right)\oplus\bigoplus_{p,q>0}\left(M^{\Box p}\underline{\ts} M^{\Box q}\right).$$
Therefore, the maps $f:T_H^c(M)\underline{\ts} T_H^c(M)\ra M$ and $g:T_H^c(M)\underline{\ts} T_H^c(M)\ra H$ are determined by their restrictions:

$$
\left\{\begin{matrix}g_{00}:H\underline{\ts} H\ra H\\ g_{0n}:H\underline{\ts} M^{\Box n}\ra H\\ g_{n0}:M^{\Box n}\underline{\ts} H\ra H\\ g_{pq}:M^{\Box p}\underline{\ts} M^{\Box q}\ra H\end{matrix}\right.\ \ \text{and} \ \ 
\left\{\begin{matrix}f_{00}:H\underline{\ts} H\ra M\\ f_{0n}:H\underline{\ts} M^{\Box n}\ra M\\ f_{n0}:M^{\Box n}\underline{\ts} H\ra M\\ f_{pq}:M^{\Box p}\underline{\ts} M^{\Box q}\ra M\end{matrix}\right. .$$

We consider the following module and comodule structures on $M^{\Box p}$: the $H$-bicomodule structure comes from two external components; the $H$-bimodule structure is induced by the inclusion $M^{\Box p}\subset M^{\ts p}$ where the latter admits the bimodule structure arising form the tensor product. Similarly, there exists an $H$-Hopf bimodule structure on $T_H^c(M)\underline{\ts} T_H^c(M)$ where the bimodule structure arises from the two external components and the bicomodule structure comes from the tensor product.

\begin{definition}
The pair $(f,g)$ with $f:T_H^c(M)\underline{\ts} T_H^c(M)\ra M$ and $g:T_H^c(M)\underline{\ts} T_H^c(M)\ra H$ is said to have property (\textbf{MB}) if
\begin{enumerate}[(MB1)]
\item $g$ is a coalgebra morphism; $f$ is an $H$-Hopf bimodule morphism with $f_{00}=0$;
\item $g$ is associative: for any $i,j,k\geq 0$, on $M^{\Box i}\underline{\ts} M^{\Box j}\underline{\ts} M^{\Box k}$ where $H=M^{\Box 0}$, we have
$$g_{0k}(g_{ij}\underline{\ts} {\id}^{\ts k})+\sum_{n=1}^{i+j}g_{nk}(f^{\ts n}\Delta^{(n)}\underline{\ts}  {\id}^{\ts k})=g_{i0}({\id}^{\ts i}\underline{\ts} g_{jk})+\sum_{m=1}^{j+k}g_{im}({\id}^{\ts i}\underline{\ts} f^{\ts m}\Delta^{(m)});$$
\item $f$ is associative: for any $i,j,k\geq 0$, on $M^{\Box i}\underline{\ts} M^{\Box j}\underline{\ts} M^{\Box k}$ where  $H=M^{\Box 0}$, we have
$$f_{0k}(g_{ij}\underline{\ts} {\id})+\sum_{n=1}^{i+j}f_{nk}(f^{\ts n}\Delta^{(n)}\underline{\ts} {\id}^{\ts k})=f_{i0}({\id}\underline{\ts} g_{jk})+\sum_{m=1}^{j+k}f_{im}({\id}^{\ts i}\underline{\ts} f^{\ts m}\Delta^{(m)}).$$
\end{enumerate}
\end{definition}
If the pair $(f,g)$ satisfies the property (\textbf{MB}), according the condition (MB1) in the definition, there exists a coalgebra map $\mu:T_H^c(M)\underline{\ts} T_H^c(M)\ra T_H^c(M)$ given by
$$\mu=g+\sum_{n\geq 1}f^{\ts n}\Delta^{(n)}$$
satisfying $\pi\circ\mu=g$ and $p\circ\mu=f$ where $\pi:T_H^c(M)\ra H$ and $p:T_H^c(M)\ra M$ are projections onto degree $0$ and $1$, respectively.
\begin{theorem}
The coalgebra map $\mu$ is associative if and only if $(f,g)$ is a pair having property (\textbf{MB}). Furthermore, with the deconcatenation coproduct, $T_H^c(M)$ is a Hopf algebra.
\end{theorem}

\begin{proof}
We start by showing that $\mu$ is associative: $\mu({\id}\underline{\ts}\mu)=\mu(\mu\underline{\ts}{\id})$. According to the universal property, it suffices to show:
$$\pi\mu(\mu\underline{\ts} {\id})=\pi\mu({\id}\underline{\ts}\mu):T_H^c(M)\underline{\ts} T_H^c(M)\underline{\ts} T_H^c(M)\ra H$$
and
$$p\mu(\mu\underline{\ts}{\id})=p\mu({\id}\underline{\ts} \mu):T_H^c(M)\underline{\ts} T_H^c(M)\underline{\ts} T_H^c(M)\ra M.$$
According to the explicit form of $\mu$:
\begin{eqnarray*}
\mu(\mu\uts{\id})&=&g(g\uts {\id})+g(\sum_{n>0} f^{\ts n}\Delta^{(n)}\uts {\id})+\\
&+& \sum_{n>0}f^{\ts n}\Delta^{(n)}(g\uts {\id})+\sum_{n>0}f^{\ts n}\Delta^{(n)}(\sum_{m>0}f^{\ts m}\Delta^{(m)}\uts{\id}).
\end{eqnarray*}
\begin{eqnarray*}
\mu({\id}\uts \mu)&=&g({\id}\uts g)+g({\id}\uts \sum_{n>0} f^{\ts n}\Delta^{(n)})+\\
&+& \sum_{n>0}f^{\ts n}\Delta^{(n)}({\id}\uts g)+\sum_{n>0}f^{\ts n}\Delta^{(n)}({\id}\uts \sum_{m>0}f^{\ts m}\Delta^{(m)}).
\end{eqnarray*}

Applying $p$ and $\pi$ gives:
\begin{eqnarray*}
\pi\mu(\mu\uts id)=g(g\uts {\id})+g(\sum_{n>0} f^{\ts n}\Delta^{(n)}\uts {\id}),\end{eqnarray*}
\begin{eqnarray*}
\pi\mu({\id}\uts \mu)=g({\id} \uts g)+g({\id}\uts \sum_{n>0} f^{\ts n}\Delta^{(n)}),
\end{eqnarray*}
\begin{eqnarray*}
p\mu(\mu\uts {\id})=p\left(\sum_{n>0}f^{\ts n}\Delta^{(n)}(g\uts {\id})+\sum_{n>0}f^{\ts n}\Delta^{(n)}(\sum_{m>0}f^{\ts m}\Delta^{(m)}\uts {\id})\right),
\end{eqnarray*}
\begin{eqnarray*}
p\mu({\id}\uts\mu)=p\left(\sum_{n>0}f^{\ts n}\Delta^{(n)}({\id}\uts g)+\sum_{n>0}f^{\ts n}\Delta^{(n)}({\id}\uts \sum_{m>0}f^{\ts m}\Delta^{(m)})\right).
\end{eqnarray*}
Once restricted to the component $M^{\Box i}\uts M^{\Box j}\uts M^{\Box k}$, the condition $\pi\mu(\mu\uts {\id})=\pi\mu({\id}\uts \mu)$ can be written as
$$g_{0k}(g_{ij}\uts {\id}^{\ts k})+\sum_{n=1}^{i+j}g_{nk}(f^{\ts n}\Delta^{(n)}\uts  {\id}^{\ts k})=g_{i0}({\id}^{\ts i}\uts g_{jk})+\sum_{m=1}^{j+k}g_{im}({\id}^{\ts i}\uts f^{\ts m}\Delta^{(m)})$$
and $p\mu(\mu\uts {\id})=p\mu({\id}\uts\mu)$ gives:
$$f_{0k}(g_{ij}\uts {\id})+\sum_{n=1}^{i+j}f_{nk}(f^{\ts n}\Delta^{(n)}\uts {\id}^{\ts k})=f_{i0}({\id}\uts g_{jk})+\sum_{m=1}^{j+k}f_{im}({\id}^{\ts i}\uts f^{\ts m}\Delta^{(m)}).$$
As the pair $(f,g)$ satisfies the property (\textbf{MB}), identities above hold according to the conditions (MB2) and (MB3). This proves the associativity of $\mu$. Since $\mu$ is a coalgebra morphism, $T_H^c(M)$ is a bialgebra.
\par
To show $T_H^c(M)$ is a Hopf algebra, we use the following lemma due to Takeuchi \cite{Tak}:

\begin{lemma}
Let $C$ be a coalgebra with $C_0=\corad(C)$ its coradical and $A$ be an algebra. Then $\Hom(C,A)$ admits the convolution product $\ast$. For any $f\in\Hom(C,A)$, $f$ is invertible under $\ast$ if and only if $f|_{C_0}$ is invertible in $\Hom(C_0,A)$.
\end{lemma}

Here we take $C=T_H^c(M)=A$, then $C_0=\corad(C)=\corad(H)$. It suffices to show ${\id}_{T_H^c(M)}$ is convolution invertible, this is clear from the lemma above as $H$ is a Hopf algebra.
\end{proof}

\begin{definition}
We call the Hopf algebra $T_H^c(M)$ obtained in the theorem above a multi-brace cotensor Hopf algebra.
\end{definition}

\begin{remark} 
In general, the Hopf algebra $T_H^c(M)$ defined above is not a graded algebra, but the coproduct is always graded.
\end{remark}

\subsection{Graded pairs}\label{Sec:Graded}

Graded pairs characterize multi-brace cotensor Hopf algebras arising from bosonizations of the cofree braided Hopf algebra \cite{JR12}.

\begin{definition}
A pair $(f,g)$ with property (\textbf{MB}) is called graded if $g_{00}=m_H$ is the multiplication in $H$ and for any $(p,q)\neq (0,0)$, $g_{pq}=0$.
\end{definition}

The following proposition gives an equivalent but more convenient description of graded pairs.

\begin{proposition}
The projection $\pi:T_H^c(M)\ra H$ onto degree $0$ is an algebra morphism if and only if the pair $(f,g)$ is graded. Moreover, if the pair $(f,g)$ is graded, $\pi$ is a morphism of Hopf algebra.
\end{proposition}

\begin{proof}
If $\pi:T_H^c(M)\ra H$ is an algebra morphism, for any $x\in M^{\Box p}$ and $y\in M^{\Box q}$, $\pi\mu(x\uts y)=\pi(x)\pi(y)$.
\par
Using the explicit form of $\mu$, the left hand side gives
\begin{eqnarray*}
\pi\mu(x\uts y) &=& \pi\left(g(x\uts y)+\sum_{n\geq 1} f^{\ts n}\Delta^{(n)}(x\uts y)\right)\\
&=& \pi(g_{pq}(x\uts y))\\
&=& g_{pq}(x\uts y).
\end{eqnarray*}
In the right hand side, if $(p,q)\neq (0,0)$, $\pi(x)\pi(y)=0$, this forces $g_{pq}=0$ for any $(p,q)\neq (0,0)$ and $g_{00}=m_H$.
\par
The rest of this proposition is clear.
\end{proof}
If the pair is graded, we can naturally find a braided Hopf algebra inside the cotensor Hopf algebra $T_H^c(M)$ using the machinery constructed by Radford \cite{Rad}.
\par
Let $(f,g)$ be a graded pair. We obtain two Hopf algebra morphisms $\pi:T_H^c(M)\ra H$ and $\iota:H\ra T_H^c(M)$ given by the projection onto and the embedding into degree $0$. It is clear that they satisfy $\pi\circ\iota={\id}_H$. Then we can use the construction of Radford on Hopf algebras with a projection to get an isomorphism of Hopf algebras
$$T_H^c(M)\cong Q\ts H$$
where 
$$Q=T_H^c(M)^{co\pi}=\{x\in T_H^c(M)|\ ({\id}\ts\pi)\Delta(x)=x\ts 1\}$$
is a braided Hopf algebra in the category ${}^H_H\mathcal{YD}$ of Yetter-Drinfel'd modules over $H$.
\par
As $M$ is a right $H$-Hopf module, we let $V=M^{coR}$ denote the set of right coinvariants in $M$. Then $V\subset Q$ is a braided vector space with the braiding $\s$ arising from the Yetter-Drinfel'd module structure. We let $Q_\s(V)$ denote the subalgebra of $Q$ generated by $V$; $Q_\s(V)$ is a braided Hopf algebra in the category ${}^H_H\mathcal{YD}$.
\par
Once a braided Hopf algebra $B$ in the category ${}^H_H\mathcal{YD}$ is given, we can form the bosonization of $B$ and $H$ (\cite{AS02}), denoted by $B\# H$: it is a Hopf algebra linearly isomorphic to $B\ts H$ with well-chosen algebra and coalgebra structures. When this construction is applied to the situation above, we obtain a Hopf algebra $Q_H(M)$ as the bosonization of $Q_\s(V)$ and $H$: it is isomorphic to the sub-Hopf algebra of $T_H^c(M)$ generated by $H$ and $M$.

\section{Examples}\label{Sec2}

\subsection{Pairs of Nichols type}
The simplest example of pairs $(f,g)$ with property (\textbf{MB}) arises from Nichols algebras.

\begin{definition}
A pair $(f,g)$ is called of Nichols type if 
\begin{enumerate}
\item $g_{00}=m_H$ and for any $(p,q)\neq (0,0)$, $g_{pq}=0$;
\item $f_{01}=a_L$, $f_{10}=a_R$ and for any $(p,q)\neq (0,1), (1,0)$, $f_{pq}=0$ where $a_L:H\uts M\ra M$ (resp. $a_R:M\uts H\ra M$) is the left (right) module structural map of $M$.
\end{enumerate}
\end{definition}

It is clear that the pair $(f,g)$ satisfies the property (\textbf{MB}) and moreover it is graded. The Hopf algebra $T_H^c(M)$ is the cotensor Hopf algebra defined in \cite{Nic78}.
\par
As the pair is graded, a braided Hopf algebra $Q_\s(V)$ can be associated to this pair as we have done in Section \ref{Sec:Graded}. This $Q_\s(V)$ is nothing but the quantum shuffle algebra $S_\s(V)$ defined in \cite{Ros98} which contains many important algebras such as symmetric algebras, exterior algebras, quantum planes and positive parts of quantum groups as examples.
\par
After the bosonization with $H$, we get the quantum symmetric algebra $S_H(M)$ defined in \cite{Ros98}.

\subsection{Pairs of multi-brace type}
When the pair $(f,g)$ is graded, the construction in Section \ref{Sec:Graded} gives the algebra defined and studied in \cite{JR12} associated to a multi-brace algebra.

\begin{definition}\label{Def:MB}
A pair $(f,g)$ is called of multi-brace type if 
\begin{enumerate}
\item $g_{00}=m_H$ and for any $(p,q)\neq (0,0)$, $g_{pq}=0$;
\item $f_{01}=a_L$, $f_{10}=a_R$ and for any $n\neq 1$, $f_{0n}=f_{n0}=0$ where $a_L:H\uts M\ra M$ (resp. $a_R:M\uts H\ra M$) is the left (right) module structural map of $M$;
\item $(f,g)$ satisfies the property (\textbf{MB}).
\end{enumerate}
\end{definition}

We simplify the expressions in the property (\textbf{MB}) with the help of conditions imposed on $g_{pq}$, $f_{n0}$ and $f_{n0}$. This gives the following equivalent definition.

\begin{proposition}\label{Lem:equiv}
$(f,g)$ is a pair of multi-brace type if and only if 
\begin{enumerate}
\item $g_{00}=m_H$ and for any $(p,q)\neq (0,0)$, $g_{pq}=0$;
\item $f_{01}=a_L$, $f_{10}=a_R$ and for any $n\neq 1$, $f_{0n}=f_{n0}=0$ where $a_L:H\uts M\ra M$ (resp. $a_R:M\uts H\ra M$) is the left (right) module structural map of $M$;
\item For any $p,q> 0$, $H$-Hopf bimodule morphisms $f_{pq}:M^{\Box p}\underline{\ts} M^{\Box q}\ra M$ passes through the quotient to give an $H$-Hopf bimodule morphism $f_{pq}:M^{\Box p}\underline{\ts}_H M^{\Box q}\ra M$;
\item $f$ is associative: for any $i,j,k\geq 1$, on $M^{\Box i}\underline{\ts} M^{\Box j}\underline{\ts} M^{\Box k}$,
$$\sum_{n=1}^{i+j}f_{nk}(f^{\ts n}\Delta^{(n)}\underline{\ts} {\id}^{\ts k})=\sum_{m=1}^{j+k}f_{im}({\id}^{\ts i}\underline{\ts} f^{\ts m}\Delta^{(m)}).$$
\end{enumerate}
\end{proposition}

\begin{proof}
Suppose that $(f,g)$ is a pair satisfying conditions (1) and (2) in Definition \ref{Def:MB}, we study the restrictions given by the property (\textbf{MB}). As $g_{00}=m_H$ and $g_{pq}=0$ for any $(p,q)\neq (0,0)$, (MB2) is equivalent to the associativity of $g_{00}=m_H$.
\par
To simplify the notation, we let (MB3)${}_{ijk}$ denote the condition restricted to $M^{\Box i}\underline{\ts}M^{\Box j}\underline{\ts}M^{\Box k}$.
\begin{enumerate}[(i)]
\item (MB3)${}_{000}$ always holds as $f_{00}=0$;
\item (MB3)${}_{010}$ is equivalent to $f_{10}(f_{01}\uts{\id}_H)=f_{01}({\id}\uts f_{10})$, which is the compatibility condition between $a_L$ and $a_R$;
\item (MB3)${}_{pq0}$ is equivalent to $f_{10}(f_{pq}\uts{\id}_H)=f_{pq}({\id}\uts f_{10})$, which says that $f_{pq}$ is a right $H$-module morphism;
\item (MB3)${}_{0pq}$ is equivalent to $f_{pq}$ is a left $H$-module morphism;
\item (MB3)${}_{p0q}$ is equivalent to $f_{pq}(f_{10}\uts{\id})=f_{pq}({\id}\uts f_{01})$, which means that $f_{pq}:M^{\Box p}\underline{\ts}M^{\Box q}\ra M$ factorizes through $M^{\Box p}\underline{\ts}_HM^{\Box q}$.
\item For $p,q,r>0$, (MB3)${}_{pqr}$ coincides with the condition (4).
 \end{enumerate}
 The equivalence of two definitions is then clear.
\end{proof}

As a pairing $(f,g)$ of multi-brace type is graded, we let $Q_\s(V)$ denote the braided Hopf algebra in the category ${}_H^H\mathcal{YD}$ naturally associated to this pair as in Section \ref{Sec:Graded}.
\par
We recall that $V=M^{coR}$ is the set of right coinvariants in $M$ which admits a braiding $\s:V\ts V\ra V\ts V$ given by 
$$\s(v\ts w)=\sum v_{(-1)}.w\ts v_{(0)}$$
for $v,w\in V$. 
\par
As $V$ is a Yetter-Drinfel'd module with the left module structure given by the adjoint action and the induced left comodule structure, the $H$-module and $H$-comodule structures on $V^{\ts n}$ are given by: for any $h\in H$ and $v^1,\cdots,v^n\in V$,
$$h.(v^1\ts\cdots\ts v^n)=\sum\textrm{ad}(h_{(1)})(v^1)\ts\cdots \ts \textrm{ad}(h_{(n)})(v^n),$$
$$\delta_L(v^1\ts\cdots\ts v^n)=\sum v^1_{(-1)}\cdots v^n_{(-1)}\ts v_{(0)}^1\ts\cdots\ts v_{(0)}^n.$$
\par
We let $\textbf{M}_{pq}$ denote the restriction of $f_{pq}$ on $V^{\ts p}\uts V^{\ts q}$, as $f_{pq}$ is a right $H$-comodule morphism, the image of $\textbf{M}_{pq}$ is contained in $V$; this gives a family of linear maps 
$$\textbf{M}_{pq}:V^{\ts p}\uts V^{\ts q}\ra V$$
and 
$$\textbf{M}=\bigoplus_{p,q\geq 0} \textbf{M}_{pq}:T(V)\uts T(V)\ra V.$$
\par

In \cite{JR12}, the authors introduced the notion of a quantum multi-brace algebra with the aim of constructing quantizations of quasi-shuffle algebras and cofree Hopf algebras. The following theorem implies that the multi-brace algebra structure can be recovered from a graded pair $(f,g)$ by taking coinvariants. Moreover, this multi-brace condition characterizes the graded pairs from those with property (\textbf{MB}). 

\begin{theorem}\label{Thm:multi-brace}
The triple $(V,\textbf{M},\s)$ is a quantum multi-brace algebra.
\end{theorem}

The rest of this subsection is devoted to giving a proof of this theorem.
\par
We recall the definition of a quantum multi-brace algebra (\cite{JR12}, Definition 4.9). Let $\chi_{ij}\in\mathfrak{S}_{i+j}$ be the permutation changing the first $i$ positions and the last $j$ positions by keeping their orders and $\beta_{ij}=T_{\chi_{ij}}$ where $T:\mc[\mathfrak{S}_n]\ra \mc[\mathfrak{B}_n]$ is the Matsumoto-Tits section (\cite{Ros98}).

\begin{definition}[\cite{JR12}]
A quantum multi-brace algebra $(V,\M,\s)$ is a braided vector space $(V, \s)$ equipped with an operation $\M=\bigoplus_{p,q\geq 0} {\M}_{pq}$, where ${\M}_{pq}:V^{\ts p}\ts V^{\ts q}\ra V$ for $p,q\geq 0$ satisfying
\begin{enumerate}
\item $\M_{00}=0$, $\M_{10}=\id_V=\M_{01}$, $\M_{n0}=0=\M_{0n}$ for $n\geq 2$;
\item for any $i,j,k\geq 1$,
$$\beta_{1k}({\M}_{ij}\otimes {\id}_V^{\otimes k})=(
{\id}_V^{\otimes k}\otimes {\M}_{ij})\beta_{i+j,k},\ \ 
\beta_{i1}({\id}_V^{\otimes i}\otimes
{\M}_{jk})=({\M}_{jk}\otimes {\id}_V^{\otimes i} )\beta_{i,j+k},
$$
\item for any triple $(i,j,k)$ of positive integers,
$$\sum_{r=1}^{i+j}{\M}_{rk}\circ ((\M^{\ts r}\circ\Delta_\beta^{(r-1)})\ts {\id}_{V}^{\ts k})=\sum_{l=1}^{j+k}{\M}_{il}\circ ({\id}_V^{\ts i}\ts (\M^{\ts l}\circ\Delta_\beta^{(l-1)})),$$
where $\Delta_\beta$ is the twisted coproduct (see Section 4 in \cite{JR12}).
\end{enumerate}
\end{definition}

For the condition (1), for any $n\neq 1$, $f_{00}=f_{n0}=f_{0n}=0$ implies that $\textbf{M}_{00}=\textbf{M}_{n0}=\textbf{M}_{0n}=0$; $\textbf{M}_{10}=\textbf{M}_{01}={\id}_V$ is clear as $f_{01}=a_L$ and $f_{10}=a_R$.
\par
Now we verify the second condition in the definition of a multi-brace algebra.

\begin{proposition}
For any $i,j,k\geq 1$, we have on $V^{\ts i}\uts V^{\ts j}\uts V^{\ts k}$:
$$\beta_{1k}(\M_{ij}\uts {\id}^{\ts k})=({\id}^{\ts k}\uts\M_{ij})\beta_{i+j,k},$$ 
$$\beta_{i1}({\id}^{\ts i}\uts\M_{jk})=(\M_{jk}\uts{\id}^{\ts i})\beta_{i,j+k}.$$
\end{proposition}

\begin{proof}
We show the first identity; the second one can be verified similarly.
\par
For any $p,q\geq 0$, $f_{pq}$ is a Hopf bimodule morphism; as its restriction on $V^{\ts p}\uts V^{\ts q}$, $\M_{pq}$ is an $H$-module and an $H$-comodule morphism where $V^{\ts p}\uts V^{\ts q}$ admits the adjoint $H$-module structure and the sub-H-comodule structure coming from $M^{\Box p}\uts M^{\Box q}$.
\par
For $i,j,k\geq 1$, we take $u^1,\cdots,u^i,v^1,\cdots,v^j,w^1,\cdots,w^k\in V$. To simplify notations, we let $(u^1,\cdots,u^i)$ denote the tensor product $u^1\ts\cdots\ts u^i$. Then the action of $\beta_{i+j,k}$ on 
$$(u^1,\cdots,u^i)\uts (v^1,\cdots,v^j)\uts (w^1,\cdots,w^k)\in V^{\ts i}\uts V^{\ts j}\uts V^{\ts k}$$
gives
\begin{eqnarray*}
& &\sum (u_{(-k)}^1\cdots u_{(-k)}^iv_{(-k)}^1\cdots v_{(-k)}^j.w^1,\cdots,u_{(-1)}^1\cdots u_{(-1)}^iv_{(-1)}^1\cdots v_{(-1)}^j.w^k)\uts\\
&\uts &(u_{(0)}^1,\cdots,u_{(0)}^i)\uts (v_{(0)}^1,\cdots,v_{(0)}^j),
\end{eqnarray*}
when $({\id}^{\ts k}\uts\M_{ij})$ is applied, we obtain
$$\sum (u_{(-1)}^1\cdots u_{(-1)}^iv_{(-1)}^1\cdots v_{(-1)}^j\cdot (w^1,\cdots,w^k))\uts \M_{ij}((u_{(0)}^1,\cdots,u_{(0)}^i)\uts (v_{(0)}^1,\cdots,v_{(0)}^j)).$$
For the left hand side, we compute
\begin{eqnarray*}
& & \delta_L(\M_{ij}((u^1,\cdots,u^i)\uts (v^1,\cdots,v^j))\\
&=& ({\id}\ts\M_{ij})\delta_L((u^1,\cdots,u^i)\uts (v^1,\cdots,v^j))\\
&=& \sum u_{(-1)}^1\cdots u_{(-1)}^iv_{(-1)}^1\cdots v_{(-1)}^j \uts \M_{ij}((u_{(0)}^1,\cdots,u_{(0)}^i)\uts (v_{(0)}^1,\cdots,v_{(0)}^j)).
\end{eqnarray*}
According to the explicit formula of the braiding, the left hand side gives the same result as the right one.
\end{proof}

The last condition on the multi-brace algebra comes from the following lemma:

\begin{lemma}
For any $i,j,k\geq 1$, on $V^{\ts i}\uts V^{\ts j}\uts V^{\ts k}$,
$$\sum_{n=1}^{i+j}\M_{nk}(\M^{\ts n}\Delta_\s^{(n)}\uts {\id}^{\ts k})=\sum_{m=1}^{j+k}\M_{im}({\id}^{\ts i}\uts \M^{\ts m}\Delta_\s^{(m)}).$$
\end{lemma}
Here $\Delta_\s:Q_\s(V)\ra Q_\s(V)\ts Q_\s(V)$ is the twisted coproduct in the braided Hopf algebra $Q_\s(V)$. This lemma comes directly from the condition (4) in Proposition \ref{Lem:equiv}.
\par
This terminates the proof of the theorem.

\subsection{Quantum quasi-symmetric algebras}\label{Sec:qshuffle}
Apart from the one of Nichols type, the simplest pair of multi-brace type is the quasi-symmetric one defined below.

\begin{definition}\label{Def:7}
A pair $(f,g)$ is called of quasi-symmetric type if 
\begin{enumerate}
\item $f_{01}=a_L$, $f_{10}=a_R$, $f_{11}=\alpha$ and for any $(p,q)\neq (0,1), (1,0), (1,1)$, $f_{pq}=0$ where $a_L:H\uts M\ra M$ (resp. $a_R:M\uts H\ra M$) is the left (right) module structural map of $M$;
\item The pair $(f,g)$ has property (\textbf{MB}) and is graded.
\end{enumerate}
\end{definition}

The following corollary is clear from Proposition \ref{Lem:equiv}. 

\begin{corollary}\label{Cor:qsym}
Let $(f,g)$ be a pair satisfying (1) as in Definition \ref{Def:7}. Then it is of quasi-symmetric type if and only if $f_{11}=\alpha:M\uts M\ra M$ satisfies:
\begin{enumerate}
\item $\alpha$ passes through $M\uts_H M$ and is a Hopf bimodule morphism;
\item $\alpha$ is associative: $\alpha(\alpha\ts \id)=\alpha(\id\ts\alpha):M\uts_H M\uts_H M\ra M$.
\end{enumerate}
\end{corollary}

If $(f,g)$ is a pair of quasi-symmetric type, we call the Hopf algebra $Q_H(M)$ associated to this pair a quantum quasi-symmetric algebra and the braided Hopf algebra $Q_\s(V)$ a quantum quasi-shuffle algebra. According to Theorem \ref{Thm:multi-brace}, it coincides with the sub-algebra generated by $V$ of the quantum quasi-shuffle algebra constructed in \cite{JR12} and \cite{JRZ11}. The Hopf algebra $Q_H(M)$ is isomorphic to the bosonization of $Q_\s(V)$ and $H$. 
\par
Let $\ast$ denote the multiplication in $Q_H(M)$. We give the explicit formula of the multiplication in the quantum quasi-symmetric algebras. As $Q_H(M)\cong Q_\s(V)\# H$ is the bosonization of $Q_\s(V)$ and $H$, the multiplication between $Q_\s(V)$ and $H$ is given by (see Section 1.5 in \cite{AS02}): for $x,y\in Q_\s(V)$ and $s,t\in H$,
$$(x\ts s)\ast (y\ts t)=\sum x\ast(s_{(1)}\cdot y)\ts s_{(2)}t.$$
The explicit formula of the multiplication in $Q_\s(V)$ can be expressed inductively according to Proposition 3 in \cite{JRZ11}. We will not copy the whole formula but write down a particular case which will be useful later: for $u,v\in V$, 
$$u\ast v=u\ts v+\s(v\ts u)+\alpha(u\ts v).$$
Moreover, if $\alpha=0$, the quantum quasi-symmetric algebra reduces to the quantum symmetric algebra and the multiplication formula is given in \cite{Ros98}.

\section{Quantum groups as quantum quasi-symmetric algebras}\label{Sec:main}

For convenience, we will use $\ts$ instead of $\underline{\ts}$ if there is no ambiguity. 

\subsection{Construction}\label{Sec:construction}

In this section, we will construct the entire quantum group as a quotient of a quantum quasi-symmetric algebra.
\par
Let $\g$ be a symmetrizable Kac-Moody Lie algebra of rank $n$ with generalized Cartan matrix $C=(c_{ij})_{n\times n}$ and $A=DC=(a_{ij})_{n\times n}$ its symmetrization with $D=\textrm{diag}(d_1,\dots,d_n)$. Let $q\in\mc^*$ not be a root of unity and $q_i=q^{d_i}$. We let $I$ denote the index set $\{1,\cdots,n\}$. 
\par
Let $H=\mc[K_1^{\pm 1},\cdots,K_n^{\pm 1}]$ be the group algebra of the additive group $\mathbb{Z}^n$. Let $W$ be the vector space generated by $\{E_i,F_i,\xi_i|\ i\in I\}$ and $M=W\ts H$.
\par
We consider the following $H$-Hopf bimodule structure on $M$: 
\begin{enumerate}
\item The right module structure arises from the multiplication in $H$; the right comodule structure comes from the tensor product of two right comodules $W$ and $H$, where the structure on $H$ is given by comultiplication and the one on $W$ is trivial (for any $w\in W$, $\delta_R(w)=w\ts 1$).
\item The left structures are given in the following way: on $W$, for any $i,j\in I$,
$$K_i.E_j=q^{a_{ij}}E_j,\ \ K_i.F_j=q^{-a_{ij}}F_j,\ \ K_i.\xi_j=\xi_j;$$
$$\delta_L(E_i)=K_i\ts E_i,\ \ \delta_L(F_i)=K_i\ts F_i,\ \ \delta_L(\xi_i)=K_i^2\ts \xi_i,$$
then we take on $M=W\ts H$ the structure arising from the tensor product.
\end{enumerate}
\par
We use the following abuse of notation: for $x\in W$ and $K\in H$, we shall write $x$ for $x\ts 1$ in $M$, and $xK$ for $x\ts K$ in $M$.
\par
We define $\alpha:M\ts M\ra M$ by: for any $K,K'\in H$, if $\lambda$ is the constant such that $K.F_j=\lambda F_j$,
$$\alpha(E_iK\ts F_jK')=\delta_{ij}\frac{\lambda\xi_iKK'}{q_i-q_i^{-1}};$$
and on any other elements not of the above form, $\alpha$ gives $0$.
We verify that $\alpha$ satisfies the conditions posed on $f_{11}$ for a pair of quasi-symmetric type.
\begin{enumerate}
\item $\alpha$ factorizes through $M\ts_H M$. It suffices to deal with its evaluation on $E_iK\ts F_jK'$ as in the definition of $\alpha$:
$$\alpha(E_i\ts K.(F_jK'))=\lambda\alpha(E_i\ts F_jKK')=\delta_{ij}\frac{\lambda\xi_iKK'}{q_i-q_i^{-1}}=\alpha(E_iK\ts F_jK').$$
\item $\alpha$ is an $H$-bimodule morphism. The right structure is clear. For the left module structure, we take some $K_p\in H$ for $p\in I$, then
$$\alpha(K_p\cdot (E_iK\ts F_jK'))=\alpha((K_p\cdot E_i)K_pK\ts F_jK')=\delta_{ij}\frac{\xi_i q^{a_{pi}-a_{pj}}\lambda K_pKK'}{q_i-q_i^{-1}}=\delta_{ij}\frac{\xi_i\lambda K_pKK'}{q_i-q_i^{-1}},$$
which coincides with $K_p\cdot \alpha(E_iK\ts F_jK')$.
\item $\alpha$ is an $H$-bicomodule morphism. We may assume that $K$ is a monomial, then $\Delta(K)=K\ts K$. For the left comodule structure,
$$(\id\ts\alpha)\delta_L: E_iK\ts F_jK'\mapsto \sum K_iK_jKK_{(1)}'\ts \frac{\delta_{ij}\lambda\xi_iKK_{(2)}'}{q_i-q_i^{-1}},$$
$$\delta_L: \delta_{ij}\lambda\xi_iKK'\mapsto \sum K_i^2KK_{(1)}'\ts\delta_{ij}\lambda\xi_iKK_{(2)}'.$$
For the right comodule structure,
$$(\alpha\ts\id)\delta_R:E_iK\ts F_jK'\mapsto \sum \frac{\delta_{ij}\lambda\xi_iKK_{(1)}'}{q_i-q_i^{-1}}\ts KK_{(2)}',$$
$$\delta_R: \delta_{ij}\lambda\xi_iKK'\mapsto \sum \delta_{ij}\lambda\xi_iKK_{(1)}'\ts KK_{(2)}'.$$
\item $\alpha$ is associative. This is clear from definition.
\end{enumerate}

According to Corollary \ref{Cor:qsym}, taking $f_{11}=\alpha$, $f_{10}=a_R$ and $f_{01}=a_L$ gives a Hopf algebra $Q_H(M)$ whose multiplication will be denoted by $\ast$. We consider the ideal $J$ of $Q_H(M)$ generated by $\{\xi_i-K_i^2+1|\ i\in I\}$, then $J$ is a Hopf ideal. Indeed, the coproduct on $\xi_i-K_i^2+1$ gives
$$\Delta(\xi_i-K_i^2+1)=K^2\ts (\xi_i-K_i^2+1)+(\xi_i-K_i^2+1)\ts 1.$$
We define $U_H(M)=Q_H(M)/J$, then $U_H(M)$ is a Hopf algebra.

\begin{remark}
The left comodule structure on $\xi_i$ is uniquely determined by $\alpha$ and the comodule structures on $E_i$, $F_i$.
\end{remark}

\subsection{Relation with quantum groups}\label{Sec:QGrp}
We show in this subsection that $U_H(M)$ is isomorphic as a Hopf algebra to the quantum group $U_q(\g)$.
\par
We start by constructing a Hopf algebra morphism from the quantum group $U_q(\g)$ to $U_H(M)$.
\par
Let $X$ be the vector space generated by $\{e_i,f_i,t_i^{\pm 1}|\ i\in I\}$. Then the quantum group $U_q(\g)$ is a quotient of the tensor algebra $T(X)$ by an ideal $R$ generated by the well-known relations. This quotient admits a unique Hopf algebra structure determined by 
$$\Delta(e_i)=t_i\ts e_i+e_i\ts 1,\ \ \Delta(f_i)=f_i\ts t_i^{-1}+1\ts f_i,\ \ \Delta(t_i^{\pm 1})=t_i^{\pm 1}\ts t_i^{\pm 1},$$
$$\ve(e_i)=0,\ \ \ve(f_i)=0,\ \ \ve(t_i^{\pm 1})=1$$
for any $i\in I$.
\par
We define a linear map $\wt{\Phi}:X\ra U_H(M)$ by
$$e_i\mapsto E_i,\ \ f_i\mapsto F_i\ast K_i^{-1}=F_iK_i^{-1},\ \ t_i^{\pm 1}\mapsto K_i^{\pm 1}.$$
The universal property of the tensor algebra gives an algebra morphism $\Phi:T(X)\ra U_H(M)$.

\begin{proposition}
The algebra morphism $\Phi$ passes through $T(X)/R$ and gives an algebra morphism $\vp:U_q(\g)\ra U_H(M)$.
\end{proposition}

\begin{proof}
We verify that $\Phi$ sends all generators of $R$ to zero. The computation below can be obtained by applying explicit formulas for the multiplication given in Section \ref{Sec:qshuffle}.
\par
It is clear that
$$\Phi(t_it_j-t_jt_i)=K_i\ast K_j-K_j\ast K_i=0,\ \ \Phi(t_it_i^{-1}-1)=K_i\ast K_i^{-1}-1=0.$$
We show that $\Phi(t_ie_j-q^{a_{ij}}e_jt_i)=K_i\ast E_j-q^{a_{ij}}E_j\ast K_i=0$ and $\Phi(t_if_j-q^{-a_{ij}}f_jt_i)=K_i\ast F_j\ast K_j^{-1}-q^{-a_{ij}}F_j\ast K_j^{-1}\ast K_i=0$. Indeed,
$$K_i\ast E_j=q^{a_{ij}}E_j\ts K_i,\ \ E_j\ast K_i=E_j\ts K_i,$$
$$K_i\ast F_j\ast K_j^{-1}=q^{-a_{ij}}F_j\ts K_iK_j^{-1},\ \ F_j\ast K_j^{-1}\ast K_i=F_j\ts K_iK_j^{-1}.$$
We consider the commutation relation 
$$\Phi\left(e_if_j-f_je_i-\delta_{ij}\frac{t_i-t_i^{-1}}{q_i-q_i^{-1}}\right)=E_i\ast F_j\ast K_j^{-1}-F_j\ast K_j^{-1}\ast E_i-\delta_{ij}\frac{K_i-K_i^{-1}}{q_i-q_i^{-1}}.$$
After an easy calculation using the algebra structure on $Q_H(M)$, we have:
$$E_i\ast F_j\ast K_j^{-1}=\delta_{ij}\xi_iK_j^{-1}+q^{-a_{ij}}F_jK_iK_j^{-1}\ts E_iK_j^{-1}+E_i\ts F_jK_j^{-1},$$
$$F_j\ast K_j^{-1}\ast E_i=E_i\ts F_jK_j^{-1}+q^{-a_{ij}}F_jK_iK_j^{-1}\ts E_iK_j^{-1},$$
where we used the fact that $A$ is a symmetric matrix. Then 
$$E_i\ast F_j\ast K_j^{-1}-F_j\ast K_j^{-1}\ast E_i=\delta_{ij}\xi_iK_j^{-1}=\delta_{ij}\frac{K_i-K_i^{-1}}{q_i-q_i^{-1}}.$$
\par
It remains to deal with the quantum Serre relations. We let $Q_H^{>0}(M)$ denote the subalgebra of $Q_H(M)$ generated by $W_+=\{E_i|\ i\in I\}$, then $M_+=W_+\ts H$ is a sub-$H$-Hopf bimodule of $M=W\ts H$ and the restriction of $\alpha=f_{11}$ on $M_+$ is zero. This shows that when restricted to $Q_H^{>0}(M)$, the quantum quasi-shuffle product degenerates to the quantum shuffle product. According to Lemma 14 in \cite{Ros98}, the quantum Serre relations on $\{e_i|\ i\in I\}$ are sent to zero by $\Phi$.
\par
The same argument can be applied to the quantum Serre relations on $f_i$ since $f_it_i^{-1}$ satisfy the same quantum Serre relations as $f_i$.
\par
Then $\Phi$ annihilates all defining relations in $U_q(\g)$ and gives an algebra morphism $\vp:U_q(\g)\ra U_H(M)$.
\end{proof}

Moreover, as $\vp$ preserves the coproduct on generators, it is a coalgebra morphism and therefore a Hopf algebra morphism. 
\par
Let $U_q^{\geq 0}(\g)$ ($U_q^{>0}(\g)$), $U_q^{\leq 0}(\g)$ ($U_q^{<0}(\g)$) and $U_q^0$ be the (strictly) positive, (strictly) negative and torus part in the quantum group $U_q(\g)$. Let $Q$ be the root lattice of the Kac-Moody Lie algebra $\g$ with simple roots $\{\alpha_1,\cdots,\alpha_n\}$. The adjoint action of $U_q^0$ on $U_q(\g)$ is diagonalizable and it decomposes $U_q(\g)$ into weight spaces indexed by $Q$: 
$$U_q(\g)=\bigoplus_{\alpha\in Q}U_q(\g)_\alpha,\ \ U_q(\g)_\alpha=\{ x\in U_q(\g)|\ K_ixK_i^{-1}=q^{(\alpha_i,\alpha)}x,\ \forall i\in I\}.$$
For $\alpha=\sum_{i=1}^n m_i\alpha_i\in Q$, we let $\hht(\alpha)=\sum_{i=1}^n |m_i|$ denote the height of $\alpha$.

\begin{proposition}
There is no non-zero Hopf ideal in $U_q(\g)$ which does not intersect $U_q(\g)_\alpha$ for $\hht(\alpha)\leq 1$.
\end{proposition}

\begin{proof}
Let $K$ be such an ideal. As $K$ is a $U_q^0$-submodule of the diagonalizable $U_q^0$-module $U_q(\g)$,
$$K=\bigoplus_{\alpha\in Q}K\cap U_q(\g)_\alpha.$$
We denote $K_+=K\cap U_q^{\geq 0}(\g)$ and $K_-=K\cap U_q^{\leq 0}(\g)$, then $K=K_+\oplus K_-$. We consider $K_-$: it is a Hopf ideal in $U_q^{\leq 0}(\g)$ which does not intersect with $U_q^0$. Moreover, in the length gradation of $U_q^{\leq 0}(\g)$ ($\textrm{deg}(f_i)=1$ and $\textrm{deg}(K_i^{\pm 1})=0$), $K_-$ is contained in the set of elements of degree no less than 2, which forces $K_-=\{0\}$ as $U_q^{<0}(\g)$ is a quantum shuffle algebra. The same argument forces $K_+=\{0\}$ and then $K=\{0\}$.
\end{proof}

\begin{theorem}
The Hopf algebra morphism $\vp:U_q(\g)\ra U_H(M)$ is an isomorphism.
\end{theorem}

\begin{proof}
The surjectivity is clear as $U_H(M)$ is generated by $E_i, F_i\ast K_i^{-1}$, $K_i^{\pm 1}$ and all relations in $U_q(\g)$ are satisfied in $U_H(M)$. 
\par
We prove the injectivity. Let $K$ be the kernel of $\vp$. It is a Hopf ideal in $U_q(\g)$ which does not intersect $U_q(\g)_\alpha$ for $\hht(\alpha)\leq 1$ by definition. Then the proposition above could be applied.
\end{proof}

\begin{remark}
In fact, we do not need the condition imposed on $q$ in the construction above. Therefore if $q$ is a primitive root of unity, $U_H(M)$ is isomorphic as a Hopf algebra to the "small" quantum group $u_q(\g)$.
\end{remark}

\section{Construction of irreducible representations}\label{Sec4}

Let $\mathcal{P}$ denote the weight lattice of the symmetrizable Kac-Moody Lie algebra $\g$ and $\mathcal{P}_+$ denote the set of dominant weights. We construct in this section the irreducible representation of highest weight $\lambda\in\mathcal{P}_+$ through a Radford pair arising from quantum quasi-symmetric algebras.

\subsection{Radford pair}
In \cite{Rad}, Radford studied Hopf algebras with a projection which offer a systematic way to construct braided Hopf algebras in some Yetter-Drinfel'd module categories.

\begin{definition}
Let $H$ and $K$ be two Hopf algebras. The pair $(H,K)$ is called a Radford pair if there exists Hopf algebra morphisms $i:H\ra K$ and $p:K\ra H$ such that $p\circ i=\id_H$.
\end{definition}

Once a Radford pair $(H,K)$ is given, the machinery in \cite{Rad} can be applied to 
yield an isomorphism of Hopf algebras $K\cong R\# H$ where $R=K^{coH}=\{k\in K|\ (\textrm{id}\ts p)\Delta(k)=k\ts 1\}$ is the set of right coinvariants. Moreover, $R$ is a subalgebra of $K$ and it admits an $H$-Yetter-Drinfel'd module structure making it a braided Hopf algebra. This permits us to form the bosonization $R\# H$.

\subsection{Construction of Radford pairs}\label{Sec:Const}

We keep notations from Section \ref{Sec:main}. 
\par
For a fixed weight $\lambda\in-\mathcal{P}_+$, we define a Hopf algebra $T=\mc[K_1^{\pm 1},\cdots,K_n^{\pm 1},K_\lambda^{\pm 1}]$ be the group algebra of the additive group $\mathbb{Z}^{n+1}$; it contains $H$ as a sub-Hopf algebra.
\par
We enlarge the vector space $W$ by adding a vector $v_\lambda$ and let $W'$ denote it. On the vector space $N=W'\ts T$, we consider the following $T$-Hopf bimodule structure:
\begin{enumerate}
\item The right $T$-Hopf module structure on $N$ is trivial.
\item The left $T$-Hopf module is determined by: on $W'$, for an element in $W$, the $T$-comodule structure coincides with the $H$-comodule structure on $W$, when restricted to $H$, the $T$-module structure is the same as the $H$-module structure on $W$; for the remaining elements, $\delta_L(v_\lambda)=K_\lambda\ts v_\lambda$,
$$K_\lambda.E_i=q^{(\lambda,\alpha_i)}E_i,\ \ K_\lambda.F_i=q^{-(\lambda,\alpha_i)}F_i,\ \ K_\lambda.\xi_i=\xi_i,$$
$$K_\lambda.v_\lambda=q^{-(\lambda,\lambda)}v_\lambda,\ \ K_i.v_\lambda=q^{-(\lambda,\alpha_i)}v_\lambda.$$
Then we take the left $T$-Hopf module structure arising from tensor product.
\end{enumerate}

We define $\alpha_N:N\ts N\ra N$ by: for any $K,K'\in T$, if $\lambda$ is the constant such that $K.F_j=\lambda F_j$,
$$\alpha_N(E_iK\ts F_jK')=\delta_{ij}\frac{\lambda\xi_iKK'}{q_i-q_i^{-1}};$$
and for any other elements not of the above form, $\alpha$ gives $0$.
\par
The same verification as in Section \ref{Sec:construction} shows that $\alpha_N$ satisfies conditions posed on $f_{11}:N\ts N\ra N$ in the definition of a quantum quasi-symmetric algebra. This gives a quantum quasi-symmetric algebra $Q_T(N)$.

\begin{remark}
It is clear from the above construction that $M'=W\ts T$ is a sub-$T$-Hopf bimodule of $N=W'\ts T$. As the restriction of $\alpha_N$ on $M'\ts M'$ has image in $M'$, we are allowed to form a third quantum quasi-symmetric algebra $Q_T(M)$. If we let $J$ be the Hopf ideal in $Q_T(M)$ generated by $\xi_i-K_i^2+1$ for $i\in I$, the Hopf algebra $U_T(M):=Q_T(M)/J$ is isomorphic to the extension of the quantum group $U_q(\g)$ by torus elements $K_\lambda^{\pm 1}$.
\end{remark}

It is clear that $Q_T(M)$ is a sub-Hopf algebra of $Q_T(N)$. We define a gradation on $Q_T(N)$ by letting $\textrm{deg}(v_\lambda)=1$ and elements in $T$ and $M$ to be of degree $0$. Then
$$Q_T(N)=\bigoplus_{n=0}^\infty Q_T(N)_{(n)},$$
where $Q_T(N)_{(n)}$ is the set of elements of degree $n$ in $Q_T(N)$, is a graded Hopf algebra with $Q_T(N)_{(0)}=Q_T(M)$.
\par
We let $i:Q_T(M)\ra Q_T(N)$ be the embedding into degree $0$ and $p:Q_T(N)\ra Q_T(M)$ be the projection onto degree $0$. Both $i$ and $p$ are Hopf algebra morphisms satisfying $p\circ i=\id$. We obtain a Radford pair $(Q_T(M),Q_T(N))$.
\par
We use the same notation $J$ to denote the ideal in $Q_T(N)$ generated by $\xi_i-K_i^2+1$ for $i\in I$. It is a Hopf ideal and the quotient $U_T(N):=Q_T(N)/J$ is a Hopf algebra. As $J$ is generated by elements of degree $0$, both $U_T(M)$ and $U_T(N)$ inherit gradations from $Q_T(M)$ and $Q_T(N)$ such that $U_T(N)_{(0)}=U_T(M)$. We obtain another Radford pair $(U_T(M),U_T(N))$.
\par
The machinery of Radford can be then applied to give an isomorphism of $U_T(M)$-Hopf bimodules
$$U_T(N)\cong U_T(N)^{coR}\ts U_T(M)$$
where $U_T(N)^{coR}$ admits a $U_T(M)$-Yetter-Drinfel'd module structure.
\par
Moreover, for each $p\in\mathbb{N}$, $U_T(N)_{(p)}$ is a sub-$U_T(M)$-Hopf bimodule of $U_T(N)$ and there is a family of isomorphisms of Hopf bimodules
$$U_T(N)_{(p)}\cong U_T(N)_{(p)}^{coR}\ts U_T(M).$$
As a consequence, we obtain a decomposition of $U_K(M)$-Yetter-Drinfel'd modules:
$$U_T(N)^{coR}=\bigoplus_{n=0}^\infty U_T(N)_{(n)}^{coR}.$$

\subsection{Construction of irreducible representations}
We study $R(1):=U_T(N)^{coR}_{(1)}$ in this subsection.
\par
We have seen in the last subsection that $R(1)$ is a $U_T(M)$-Yetter-Drinfel'd module where the $U_T(M)$-module structure is given by the adjoint action. As $U_T(M)$ is isomorphic to an extension of $U_q(\g)$, we have an embedding of Hopf algebra $U_q(\g)\ra U_T(M)$. Then $R(1)$ admits an adjoint $U_q(\g)$-module structure.
\par
For $\mu\in\mathcal{P}_+$, we let $L(\mu)$ denote the simple $U_q(\g)$-module of highest weight $\mu$ of type $1$.

\begin{theorem}\label{Thm:Rep}
As left $U_q(\g)$-modules, $R(1)$ is isomorphic to $L(-\lambda)$.
\end{theorem}

We reproduce the proof of the following lemma in \cite{Ros12} for convenience.

\begin{lemma}\label{Lemma}
We have $R(1)=\ad(U_T(M))(v_\lambda)$.
\end{lemma}

\begin{proof}
It is clear that $\ad(U_T(M))(v_\lambda)\subset R(1)$. In the structural theorem of Hopf modules
$$U_T(N)_{(1)}\cong R(1)\ts U_T(M),$$
where the projector $P:U_T(N)_{(1)}\ra R(1)$ has the form $P(x)=\sum x_{(0)}S(x_{(1)})$ for $\delta_R(x)=\sum x_{(0)}\ts x_{(1)}$.
\par
Elements in $U_T(N)_{(1)}$ are linear combinations of those of form $av_\lambda b$ for $a,b\in U_T(M)$. Since $v_\lambda$ is a right coinvariant,
$$P(av_\lambda b)=\sum a_{(1)}v_\lambda b_{(1)} S(b_{(2)})S(a_{(2)})=\ve(b)\sum a_{(1)} v_\lambda S(a_{(2)})=\ve(b)\ad(a)(v_\lambda).$$
This shows the other inclusion.
\end{proof}

As $K_\lambda^{\pm 1}$ acts as a scalar on $v_\lambda$, $R(1)=\ad(U_H(M))(v_\lambda)$.

\begin{proof}[Proof of theorem]
We compute the action of $E_i$ and $K_i$ on $v_\lambda$: after definition, 
$$\ad (E_i)(v_\lambda)=E_i\ast v_\lambda-q^{-(\lambda,\alpha_i)}v_\lambda\ast E_i=0,$$
$$\ad (K_i)(v_\lambda)=q^{-(\lambda,\alpha_i)}v_\lambda.$$
As it is well-known that $L(-\lambda)\cong U_q(\g)/J_{-\lambda}$ where $J_{-\lambda}$ is the left ideal of $U_q(\g)$ generated by $E_i,K_i-q^{-(\lambda,\alpha_i)}.1$ and $F_i^{1-(\lambda,\alpha_i)}$ for $i\in I$, it suffices to show that
$$\ad(F_i)^{1-(\lambda,\alpha_i)}(v_\lambda)=0.$$
This comes from the computation in Lemma 14 of \cite{Ros98}.
\par
As a consequence, we obtain a $U_q(\g)$-module surjection $L(-\lambda)\ra R(1)$. It is clearly an isomorphism, after the irreducibility of $L(-\lambda)$.
\end{proof}

\subsection{Simple modules over the double}
In the construction of irreducible representations above, only the $U_T(M)$-module structure is taken into consideration. As $R(1)$ admits a $U_T(M)$-Yetter-Drinfel'd module structure, $L(-\lambda)$ is an irreducible $U_T(M)$-Yetter-Drinfel'd module. 
\par
To simplify notations, we denote $U_T:=U_T(M)$ and $U_T^\circ$ its Hopf dual (\cite{Sweedler}, Chapter VI). We let $\mathcal{D}(U_T,U_T^\circ)$ denote the quantum double of them, then $L(-\lambda)$ is an irreducible $\mathcal{D}(U_T,U_T^\circ)$-module.
\par
In fact, we may vary the $T$-module structure on $W'$ to get another family of simple modules over the double $\mathcal{D}(U_T,U_T^\circ)$. These modules are first studied by Joseph-Letzter \cite{JL94} in a separation of variable theorem on quantum groups and then used by Baumann-Schmitt \cite{BS98} and others in classifying bicovariant differential calculi on quantum groups. The simplicity of such modules is proved in \cite{Jos99}.
\par
We preserve notations in Section \ref{Sec:main}. For a fixed weight $\lambda\in-\mathcal{P}_+$, we take the Hopf algebra $T=\mc[K_1^{\pm 1},\cdots,K_n^{\pm 1},K_\lambda^{\pm 1}]$ as in Section \ref{Sec:Const}, enlarge the vector space $W$ by $w_\lambda$ and denote it by $W_\lambda'$. The $T$-Hopf module structure on $N_\lambda=W_\lambda'\ts T$ is defined the same as in Section \ref{Sec:Const} except replacing $v_\lambda$ by $w_\lambda$, $K_\lambda.w_\lambda=w_\lambda$ and $K_i.w_\lambda=w_\lambda$. 
\par
The machinery of quantum quasi-shuffle algebra produces $Q_T(N_\lambda)$ and the quotient by the Hopf ideal generated by $\xi_i-K_i^2+1$ for $i\in I$ gives a Hopf algebra $U_T(N_\lambda)$. The construction of Radford pair $(U_T(M), U_T(N_\lambda))$ gives a $U_T(M)$-Yetter-Drinfel'd module 
$$R_\lambda(1)=U_T(N_\lambda)_{(1)}^{coR}.$$
\par
For $\mu\in -\mathcal{P}_+$, we let $L(\mu)^\sharp$ denote the irreducible representation of $U_q(\g)$ of lowest weight $\mu$.

\begin{theorem}
\begin{enumerate}
\item $R_{2\lambda}(1)$ is isomorphic to $L(-\lambda)\ts L(\lambda)^\sharp$ as a $U_T(M)$-module.
\item $R_\lambda(1)$ is a simple $U_T(M)$-Yetter-Drinfel'd module.
\end{enumerate}
\end{theorem}

\begin{proof}
\begin{enumerate}
\item We start by computing the adjoint action of $U_T(M)$ on $w_{2\lambda}$. The braiding between $E_i$, $F_i$ and $w_{2\lambda}$ is given by:
$$\s(E_i\ts w_{2\lambda})=w_{2\lambda}\ts E_i,\ \ \s(w_{2\lambda}\ts E_i)=q^{2(\lambda,\alpha_i)}E_i\ts w_{2\lambda},$$ 
$$\s(E_i\ts E_i)=q^{(\alpha_i,\alpha_i)}E_i\ts E_i,\ \ 
\s(F_i\ts w_{2\lambda})=w_{2\lambda}\ts F_i,$$
$$\s(w_{2\lambda}\ts F_i)=q^{-2(\lambda,\alpha_i)}F_i\ts w_{2\lambda},\ \ \s(F_i\ts F_i)=q^{-(\alpha_i,\alpha_i)}F_i\ts F_i$$
Then Lemma 14 in \cite{Ros98} can be applied and it gives
$$\ad(E_i)^n(w_{2\lambda})=\prod_{k=1}^n\left(\frac{q^{ka_{ii}}-1}{q^{a_{ii}}-1}\right)\prod_{k=0}^{n-1}(1-q^{k(\alpha_i,\alpha_i)}q^{2(\lambda,\alpha_i)})E_i^{\ts n}\ts w_{2\lambda},$$
$$\ad(F_i)^n(w_{2\lambda})=\prod_{k=1}^n\left(\frac{q^{ka_{ii}}-1}{q^{a_{ii}}-1}\right)\prod_{k=0}^{n-1}(1-q^{-k(\alpha_i,\alpha_i)}q^{-2(\lambda,\alpha_i)})E_i^{\ts n}\ts w_{2\lambda}.$$
As $\lambda\in-\mathcal{P}_+$, for any $i\in I$, $\lambda_i=\frac{2(\lambda,\alpha_i)}{(\alpha_i,\alpha_i)}\leq 0$ is an integer. This implies that 
$$\ad(E_i)^{-\lambda_i+1}(w_{2\lambda})=0,\ \ \ad(F_i)^{-\lambda_i+1}(w_{2\lambda})=0.$$
Moreover, the adjoint action of $K_i$ on $w_{2\lambda}$ makes it invariant. 
\par
According to Proposition 1.27 in \cite{APW91}, we obtain a $U_T(M)$-module surjection
$$\vp:L(-\lambda)\ts L(\lambda)^\sharp\ra R_{2\lambda}(1).$$
It is shown in \cite{JL94} and \cite{Cal93} that $\ad(U_q)(K_{2\lambda})\cong L(-\lambda)\ts L(-\lambda)^*$; since the $U_q(\g)$-module structure on $w_{2\lambda}$ coincides with the one on $K_{2\lambda}$, $R_{2\lambda}(1)$ is isomorphic to $L(-\lambda)\ts L(-\lambda)^\sharp$. Then $\vp$ is an isomorphism by comparing the dimension of both sides.
\item Let $L\subset R_\lambda(1)$ be a non-zero sub-$U_T(M)$-Yetter-Drinfel'd module. We show that $L$ contains $w_\lambda$, then $L=R_\lambda(1)$ after Lemma \ref{Lemma}.
\par
We take $l\in L$, according to Lemma \ref{Lemma}, $l=\ad(x)(w_\lambda)$ for some $x\in U_T(M)$. As $T$ acts diagonally on $R_\lambda(1)$, we may suppose that elements in $T$ do not appear in $x$. From the definition of the coproduct, the only component in $(\id\ts\Delta)\Delta(x)$ contained in $U_T(M)\ts\mc\ts U_T(M)$ is $x\ts 1\ts 1$. Then the compatibility condition of the module and comodule structures in a Yetter-Drinfel'd module ensures the existence of a component $u\ts w_\lambda$ in $\delta_L(l)=\delta_L(\ad(x)(w_\lambda))$ for some nonzero $u\in U_T(M)$. As $L$ is a left $U_T(M)$-comodule, $w_\lambda\in L$.
\end{enumerate}
\end{proof}
The proof of the point (1) in the theorem explains the reason of the appearance of the constant $2$ in the isomorphism.
\par
We explain its relation with the work of Baumann-Schmitt and A. Joseph cited above.
\par
We enlarge the torus part of the quantum group $U_q(\g)$ by its weight lattice, this gives a Hopf algebra $\wt{U_q}(\g)$. It acts on itself by adjoint action; we let $\mathcal{F}(\wt{U_q})$ denote the set of ad-finite elements in $\wt{U_q}(\g)$.
It is shown by Joseph-Letzter \cite{JL94} and Caldero \cite{Cal93} that there exists isomorphisms of $\wt{U_q}(\g)$-modules:
$$\ff(\wt{U_q})\cong \bigoplus_{\lambda\in\mathcal{P}_+}\End(L(\lambda))\cong \bigoplus_{\lambda\in\mathcal{P}_+}\ad(\wt{U_q})(K_{-2\lambda}).$$
\par
In the work of Baumann-Schmitt, they used the fact that $\ff(\wt{U_q})$ is a left coideal in $\wt{U_q}(\g)$. This is clear in our construction.
\par
In \cite{Jos99}, A. Joseph proved that as $\ad(\wt{U_q}(\g))(K_{-2\lambda})$ admits a $\mathcal{D}(\wt{U_q}(\g),\wt{U_q}(\g)^\circ)$-module structure and it is irreducible (Corollary 3.5, \textit{loc.cit}) by studying carefully an ad-invariant bilinear form on $\wt{U_q}(\g)$. Our construction provides another point of view and a simple proof of this result.

\section{Uniqueness of Hopf ideal}\label{Sec5}

The quantum group $U_H(M)$ is the quotient of the quantum quasi-symmetric algebra $Q_H(M)$ by the ideal generated by $\xi_i-K_i^2+1$ for $i\in I$. It is natural to ask that under which conditions on $P_i(x)\in\mc[x]$ for $i\in I$, the ideal $J$ generated by $\xi_i+P_i(K_i)$ for $i\in I$ is a Hopf ideal.
\par
\begin{proposition}
The ideal $J$ is a Hopf ideal if and only if for each $i\in I$, there exists $\lambda_i\in\mc$ such that $P_i=\lambda_i(K_i^2-1)$.
\end{proposition}

We start with the following general lemma:

\begin{lemma}
Let $\mc[G]$ be the group algebra of a group $G$ and $x=\sum\lambda_ig_i\in\mc[G]$ with $g_i\in G$ and $\lambda_i\in\mc$ such that $\Delta(x)=g\ts x+x\ts h$ for some $g\neq h\in G$. Then there exists $\lambda\in\mc$ such that $x=\lambda(g-h)$.
\end{lemma}

\begin{proof}
We consider the sub-coalgebra of $\mc[G]$ generated by $g,h$ and $x$ which is cocommutative as $\mc[G]$ is. This forces $(g-h)\ts x=x\ts (g-h)$; as $g\neq h$, $x$ is proportional to $g-h$.
\end{proof}

\begin{proof}[Proof of Proposition]
The "if" part is clear. We suppose that $J$ is a Hopf ideal, then the canonical projection $\pi:Q_H(M)\ra Q_H(M)/J$ is a morphism of Hopf algebra. It is clear from definition that $\Delta(\xi_i)=K_i^2\ts \xi_i+\xi_i\ts 1$ and applying $\pi$ gives $\Delta(P_i(K_i))=K_i^2\ts P_i(K_i)+P_i(K_i)\ts 1$ in $\pi(H)$. Since the sub-Hopf algebra of $Q_H(M)$ generated by $\xi_i$ and $K_i^{\pm 1}$ for $i\in I$ is the polynomial algebra $\mc[\xi_1,\cdots,\xi_n,K_1^{\pm 1},\cdots,K_n^{\pm 1}]$, $\pi(H)$ coincides with $H$. The "only if" part holds according to the above lemma.
\end{proof}

In $Q_H(M)$, the commutator of $E_i$ and $F_j$ reads
$$E_i\ast F_j-q^{-a_{ij}}F_j\ast E_i=\delta_{ij}\frac{\xi_i}{q_i-q_i^{-1}}.$$
If the isomorphism constructed in Section \ref{Sec:QGrp} is under consideration, the specialization $\xi_i\mapsto P_i(K_i)$ gives 
$$[E_i,F_j]=\delta_{ij}\frac{P_i(K_i)K_i^{-1}}{q_i-q_i^{-1}}.$$
The proposition above affirms that if the quotient algebra is demanded to admit a Hopf algebra structure, we have no choice other than $P_i(K_i)=K_i^2-1$ up to scaling. It shows the rigidity on the deformation of commutation relations in the framework of Hopf algebras.

\end{document}